\documentclass{amsart}
\usepackage{amssymb, amsmath, amscd, latexsym, mathrsfs, appendix}
\usepackage{graphics}
\usepackage{amsmath}
\usepackage{amsxtra}
\usepackage{amsfonts}
\usepackage{appendix}
\usepackage[all]{xy}

\usepackage{color}

\usepackage{tikz}
\usetikzlibrary{matrix,arrows}

\newtheorem{mainthm}{Theorem}[]
\numberwithin{equation}{section}
\newtheorem{thm}{Theorem}[section]
\newtheorem{cor}[thm]{Corollary}
\newtheorem{lem}[thm]{Lemma}

\newtheorem{prop}[thm]{Proposition}
\newtheorem{defn}[thm]{Definition}

\newtheorem{rem}[thm]{Remark}
\newtheorem{expl}[thm]{Example}

\newcommand{\lra}{\longrightarrow}

\newcommand{\co}{\colon\!}

\newcommand{\overcat}{/} 

\newcommand{\heq}{\textup{he}}

\newcommand{\id}{\textup{id}}

\newcommand{\map}{\textup{map}}

\newcommand{\TOP}{\textup{TOP}}

\newcommand{\emb}{\textup{emb}}
\newcommand{\config}{\mathsf{con}} 

\newcommand{\fin}{\mathsf{Fin}}
\newcommand{\boxfin}{\mathsf{Boxfin}}

\newcommand{\simp}{\mathsf{simp}}

\newcommand{\sA}{\mathcal A}
\newcommand{\sB}{\mathcal B}
\newcommand{\sC}{\mathcal C}

\newcommand{\sE}{\mathcal E}

\newcommand{\sP}{\mathcal P}

\newcommand{\LL}{\mathbb L}

\newcommand{\RR}{\mathbb R}

\newcommand{\colimsub}[1]{\begin{array}[t]{cc} \textup{colim} \\
[-1.7mm] \scriptstyle{#1} \end{array}}

\newcommand{\hocolimsub}[1]{\begin{array}[t]{cc} \textup{hocolim} \\
[-1.2mm] \scriptstyle{#1} \end{array}}

\newcommand{\uli}{\underline}

\newcommand{\pre}{\textup{pre}}

\begin{document}

\title{The configuration category of a product}
\author{Pedro Boavida de Brito and Michael S. Weiss}%

\address{Dept. of Mathematics, IST, Univ. of Lisbon, Av. Rovisco Pais, Lisboa, Portugal}%
\email{pedrobbrito@tecnico.ulisboa.pt}

\address{Math. Institut, WWU M\"{u}nster, 48149 M\"{u}nster, Einsteinstrasse 62, Germany}%
 \email{m.weiss@uni-muenster.de}

\thanks{M.W. was supported by the Bundesministerium f\"ur Bildung und Forschung through the A.v.Humboldt
foundation (Humboldt professorship 2012-2017). P.B. was supported by FCT through grant SFRH/BPD/99841/2014.}

\begin{abstract} A construction related to the Boardman-Vogt tensor product of operads allows us to describe
the configuration category of a product manifold $M\times N$ in terms of the configuration categories of the factors $M$ and $N$.
\end{abstract}
\maketitle

The configuration category of a manifold $M$ is a Segal space $\config(M)$ which aggregates the configuration space data of the manifold \cite{Andrade},  \cite{BoavidaWeissLong}. Its objects are configurations of points in $M$ (of any finite cardinality, and possibly zero); morphisms are paths of configurations in which points may fuse. More formally, a morphism from a $k$-point configuration $x$ to an $\ell$-point configuration $y$ is a pair $(f, H)$ where $f : \{1, \dots, k\} \to \{1, \dots, \ell\}$ is a map of finite sets and $H$ is a path in $M^{\times k}$ satisfying:
\begin{enumerate}
\item $H_0 = x$ and $H_1 = yf$,
\item If two components of $H_T$ agree for some $T$, then they agree for all $t \geq T$.
\end{enumerate}
The composition is given by composing functions and concatenating (Moore) paths. Different models for the configuration category of a manifold were given in \cite{BoavidaWeissLong}. The one above -- we call it the particle model -- is due to \cite{Andrade} and it has the advantage that it does not require the existence of a smooth structure to be defined.

In \cite{BoavidaWeissLong}, we described a way in which configuration categories allow for enlightening homotopical descriptions of spaces of smooth embeddings in high codimension. In this paper, we study the relation between configuration categories and products. The result is a description of the configuration
category of a product manifold in terms of the configuration categories of its factors. More precisely, we construct a bifunctor
$\boxtimes$ on the category of simplicial spaces over $\fin$, the nerve of the category of finite sets, and prove the following.

\begin{mainthm}\label{thm:main}
There is a natural weak equivalence of Segal spaces
$$\config(M) \boxtimes \config(N) \to \config(M \times N)$$
over $\fin$, for any two topological manifolds $M$ and $N$.
\end{mainthm}

When $M=\RR^m$ and $N=\RR^n$, this theorem is reminiscent of the additivity theorem of Dunn \cite{Dunn}, Fiedorowicz-Vogt \cite{FieVo} and Lurie \cite{Lurie}. We comment on the relation between the two theorems at the end of the paper. For the purposes of this introduction, we just point out that the space of homotopy automorphisms of $\config(\RR^m)$ can be identified with the space of homotopy automorphisms of $E_m$~, the operad of little $m$-disks \cite[\S7]{BoavidaWeissLong}. As a result there is a form of friendly competition going on between the configuration category point of view and the operadic point of view. The configuration category point of view has some advantages. There is an obvious action of $\TOP(m)$, the topological group of homeomorphisms $\RR^m\to \RR^m$, on $\config(\RR^m)$, whereas it is not as easy to describe a corresponding action of $\TOP(m)$ on $E_m$\,.

Moreover, theorem~\ref{thm:main} suggests that these actions of $\TOP(m)$ on $\config(\RR^m)$ for various $m$ satisfy
product compatibility. More generally still, for a topological manifold $M$, the homeomorphism group $\TOP(M)$ acts on $\config(M)$ and theorem~\ref{thm:main} suggests that these actions satisfy product compatibility. This is indeed the case. There is a convenient formulation using homotopy orbit constructions, theorem \ref{thm:mainorb} below, which turns out to be equivalent to theorem \ref{thm:main}. (The homotopy orbit constructions are to be performed degreewise; $\config(M)$ is a simplicial space and $\TOP(M)$ acts on it preserving simplicial degree.)

\begin{mainthm}\label{thm:mainorb}
There is a natural weak equivalence of Segal spaces
$$\config(M)_{h\TOP(M)} \boxtimes \config(N)_{h\TOP(N)} \lra \config(M \times N)_{h(\TOP(M)\times\TOP(N))}$$
over $\fin$ and over $B(\TOP(M)\times\TOP(N))$, for any two topological manifolds $M$ and $N$.
\end{mainthm}

\section{Boxes} \label{sec-boxes}
As in \cite{BoavidaWeissLong}, let $\fin$ be the category
whose objects are the finite sets $\uli k=\{1,2,\dots,k\}$, where $k=0,1,2,\dots$. The morphisms from $\uli k$ to $\uli \ell$
are the maps from $\uli k$ to $\uli\ell$\,, not required to be order preserving.

\begin{defn} {\rm A surjective map $f\co\uli k \to \uli \ell$ is \emph{selfic} if the injective map $\uli\ell\to \uli k$
defined by $x\mapsto \min(f^{-1}(x))\in \uli k$ is increasing.
}
\end{defn}

\begin{rem} {\rm Let $\eta$ be an equivalence relation on $\uli k$ with $\ell$ distinct (nonempty) equivalence classes.
Then there is a unique selfic map $f\co\uli k \to \uli\ell$ such that the equivalence classes of $\eta$ are
precisely the sets $f^{-1}(t)$ where $t=1,2,\dots,\ell$. }
\end{rem}

\medskip
Let $\boxfin$ be the following (discrete) small category. An object of $\boxfin$
is a diagram of the shape $\uli r\leftarrow \uli k\to \uli s\,$ in $\fin$
where the two arrows are selfic (surjective) maps
and the resulting map $\uli k\to \uli r\times \uli s$ is \emph{injective}. A morphism
in $\boxfin$ is a natural transformation between such diagrams in $\fin$. There are forgetful functors
\[ p_0,p_1,p_2\co \boxfin\to\fin \]
taking the object $\uli r\leftarrow \uli k\to \uli s\,$ above to $\uli k$\,, $\uli r$ and $\uli s$ respectively.

The inclusion of $\boxfin$ into the analogous category obtained by replacing the word \emph{selfic} by the word \emph{surjective} is an equivalence of categories. The advantage of defining $\boxfin$ using selfic maps is that it guarantees that the map $p_0$ makes $N\boxfin$ into a fiberwise complete Segal space over $N\fin$. Recall from \cite{BoavidaWeissLong} that a simplicial space $Y$ is a \emph{fiberwise complete Segal space} over a Segal space $B$ if the square
\[
	\begin{tikzpicture}[descr/.style={fill=white}]
	\matrix(m)[matrix of math nodes, row sep=2.5em, column sep=2.5em,
	text height=1.5ex, text depth=0.25ex]
	{
	Y_1^{\heq} & B_1^{\heq} \\
	Y_0 & B_0 \\
	};
	\path[->,font=\scriptsize]
		(m-1-1) edge node [auto] {} (m-1-2)
		(m-2-1) edge node [auto] {} (m-2-2)
		(m-1-1) edge node [left] {$d_1$} (m-2-1)
		(m-1-2) edge node [auto] {$d_1$} (m-2-2);
	\end{tikzpicture}
\]
is homotopy cartesian, where the superscript $\heq$ stands for \emph{homotopy equivalences}. Fiberwise complete Segal spaces are the fibrations between (not necessarily complete) Segal spaces in Rezk's model structure.

\begin{defn} \label{defn-pretensor} {\rm Let $X$ and $Y$ be fiberwise complete Segal spaces over $N\fin$.
The simplicial space $X\boxtimes^\pre Y$ is the (levelwise) pullback of
\[
	\begin{tikzpicture}[descr/.style={fill=white}]
	\matrix(m)[matrix of math nodes, row sep=2.5em, column sep=2.5em,
	text height=1.5ex, text depth=0.25ex]
	{
		 & N\boxfin  \\
	X\times Y & N\fin\times N\fin \\
	};
	\path[->,font=\scriptsize]
		(m-2-1) edge node [auto] {} (m-2-2)
		(m-1-2) edge node [auto] {$(p_1,p_2)$} (m-2-2);
	\end{tikzpicture}
\]
It comes with a preferred reference map to $N\fin$ which is the composition
\[
X\boxtimes^\pre Y \xrightarrow{\quad} N\boxfin \xrightarrow{\; p_0 \;} N\fin~.
\]
}
\end{defn}

\begin{prop} \label{prop-preboxcomplete} $X\boxtimes^\pre Y$ is a fiberwise complete Segal space over $N\fin$.
\end{prop}

\proof A homotopy pullback of Segal spaces is a Segal space. Therefore $X\boxtimes^\pre Y$ is a
Segal space. (We are using the fact that $(p_1,p_2)$ from $N\boxfin$ to $N\fin\times N\fin$ is a degreewise
fibration. This is trivial since $N\fin\times N\fin$ is a simplicial set alias simplicial discrete space.)

If, in a homotopy pullback square of Segal spaces
\[
	\begin{tikzpicture}[descr/.style={fill=white}]
	\matrix(m)[matrix of math nodes, row sep=2.5em, column sep=2.5em,
	text height=1.5ex, text depth=0.25ex]
	{
	A & B \\
	C & D \\
	};
	\path[->,font=\scriptsize]
		(m-1-1) edge node [auto] {} (m-1-2)
		(m-2-1) edge node [auto] {} (m-2-2)
		(m-1-1) edge node [left] {} (m-2-1)
		(m-1-2) edge node [auto] {} (m-2-2);
	\end{tikzpicture}
\]
the lower horizontal arrow makes $C$ fiberwise complete over $D$, then the upper horizontal arrow
makes $A$ fiberwise complete over $B$. This implies that $X\boxtimes^\pre Y$ is fiberwise complete
over $N\boxfin$ (take $C:=X\times Y$ and $D=N\fin\times N\fin$). Since $N\boxfin$ is fiberwise
complete over $N\fin$ via $p_0$~, it follows that $X\boxtimes^\pre Y$ is fiberwise complete over
$N\fin$. \qed

\begin{expl} \label{expl-config} {\rm Take $X=\config(M)$ and $Y=\config(N)$,
where $M$ and $N$ are topological manifolds. We use the particle models, so that $X=N\sA$ and $Y=N\sB$ where $\sA$ and $\sB$ are the topological categories defined in the introduction. Then $X\boxtimes^\pre Y=N(\sA\bar\times \sB)$ where an object of $\sA\bar\times\sB$ is a triple consisting of injections
$f\co \uli r\to M$, $g\co\uli s\to N$ and
$u\co \uli k\to \uli r\times\uli s$ such that the coordinates $\uli k\to \uli s$ and $\uli k\to \uli r$ of $u$ are selfic (surjective) maps.
These triples form a space in a fairly obvious way.
A morphism in $\sA\bar\times\sB$, say from an object $(f^0,g^0,u^0)$ to an object $(f^1,g^1,u^1)$,
is a triple consisting of a morphism from $f^0$ to $f^1$ in $\sA$,
a morphism from $g^0$ to $g^1$ in $\sB$ and a map
\[ \uli k^0 \lra \uli k^1 \]
over $\uli r^0\times \uli s^0 \lra \uli r^1\times \uli s^1$ (where $\uli k^j$, $\uli r^j$, $\uli s^j$ are the
sources of $u^j,f^j,g^j$ respectively, for $j=0,1$).
}
\end{expl}

\begin{rem} \label{rem-comma} {\rm For fiberwise complete Segal spaces $X$ and $Y$ over $N\fin$ as above,
let $W=X\boxtimes^\pre Y$. Select an object
\[  w\in W_0 \]
and let $x\in X_0$ and $y\in Y_0$ be the coordinates of $w$. We wish to relate the comma constructions $X\overcat x$ and
$Y\overcat y$ to $W\overcat w$. Clearly $X\overcat x$, $Y\overcat y$ and $W\overcat w$ are again fiberwise
complete Segal spaces
over $N\fin$ via the forgetful maps to $X$, $Y$ and $W$, respectively. It turns out that $W\overcat w$ is a full Segal subspace of
$X\overcat x \boxtimes^\pre Y\overcat y$, determined as such by a subset of $\pi_0$ of $\big( (X\overcat x) \boxtimes^\pre (Y\overcat y) \big)_0$.
To characterize the objects which belong to that full Segal subspace,
we only need to know the object of $\boxfin$ covered by $w$; call it $\lambda$.
An object
\[ u~\in~(X\overcat x)\boxtimes^\pre(Y\overcat y) \]
also covers an object in $\boxfin$, call it $\kappa$,
but in addition this comes with a distinguished pair of morphisms $p_1(\kappa)\to p_1(\lambda)$ and
$p_2(\kappa)\to p_2(\lambda)$ in $\fin$. That pair of morphisms in $\fin$ has at most one lift (choice of a matching morphism from $p_0(\kappa)$ to $p_0(\lambda)$ in $\fin$) to a single morphism $\kappa\to \lambda$ in $\boxfin$. The object $u$ of $(X\overcat x)\boxtimes^\pre(Y\overcat y)$ belongs to $W\overcat w$ if and only if such a lift exists.

Furthermore, suppose that $u\to u'$ is a morphism in $(X\overcat x)\boxtimes^\pre(Y\overcat y)$
which covers a morphism $\kappa\to \kappa'$ in $\boxfin$. If the induced morphism $p_0(\kappa)\to p_0(\kappa')$
is an isomorphism in $\fin$, then $u$ belongs to $W\overcat w$ if and only if $u'$ does.
}
\end{rem}

\section{Conservatization}
\begin{defn} \cite[\S8]{BoavidaWeissLong} \cite{GalToKo} {\rm Let $w\co A\to B$ be a map of simplicial spaces. We say that $w$ is \emph{conservative}
(or that $A$ is conservative over $B$) if,
for every monotone surjection $u\co [\ell]\to [k]$ in $\Delta$, the following diagram is weakly homotopy cartesian:
\[
	\begin{tikzpicture}[descr/.style={fill=white}]
	\matrix(m)[matrix of math nodes, row sep=2.5em, column sep=2.5em,
	text height=1.5ex, text depth=0.25ex]
	{
	A_k & A_\ell \\
	B_k & B_\ell \\
	};
	\path[->,font=\scriptsize]
		(m-1-1) edge node [auto] {$u^*$} (m-1-2)
		(m-2-1) edge node [auto] {$u^*$} (m-2-2)
		(m-1-1) edge node [left] {$w$} (m-2-1)
		(m-1-2) edge node [auto] {$w$} (m-2-2);
	\end{tikzpicture}
\]
}
\end{defn}

Let $\sC$ be a small (discrete) category and let $X$ be a fiberwise complete Segal space over the nerve $N\sC$,
with reference map $w\co X\to N\sC$.

\begin{lem} \label{lem-hensel} The following conditions on $w\co X\to N\sC$ are equivalent.
\begin{itemize}
\item[a)] An element $\gamma\in X_1$ is weakly invertible if and only if $w(\gamma)\in (N\sC)_1$ is an isomorphism.
\item[b)] $w$ is conservative.
\end{itemize}
\end{lem}

\proof See \cite[\S8]{BoavidaWeissLong}.

\medskip
A formula for a localization procedure to ``make'' simplicial spaces over $B$ conservative over $B$ is given in \cite[\S8]{BoavidaWeissLong} in the case
where $B$ is a simplicial \emph{set}. The name of the procedure is \emph{conservatization} and if the input is $w\co A\to B$, then
the output of the procedure is a commutative diagram of simplicial spaces
\begin{equation} \label{eqn-faceloc}
\begin{tikzpicture}[descr/.style={fill=white}, baseline=(current bounding box.base)],
	\matrix(m)[matrix of math nodes, row sep=2.0em, column sep=2.0em,
	text height=1.5ex, text depth=0.25ex]
	{
	 A & A^! & \Lambda A \\
	  & B & \\
	};
	\path[->,font=\scriptsize]
		(m-1-2) edge node [above] {$\simeq$} (m-1-1)
		(m-1-2) edge node [right] {} (m-1-3)
		(m-1-1) edge node [left] {$w$} (m-2-2)
		(m-1-2) edge node [above] {} (m-2-2)
		(m-1-3) edge node [auto] {} (m-2-2);
\end{tikzpicture}
\end{equation}
where $\Lambda A$ is conservative over $B$. The map from
$A^!$ to $A$ is a degreewise weak equivalence. More precisely,
$A\mapsto \Lambda A$ and $A\mapsto A^!$ are endofunctors of the category of simplicial spaces over $B$. The
horizontal arrows in diagram~(\ref{eqn-faceloc}) are natural transformations.

\begin{defn} \label{defn-faceloc} {\rm The formula for $\Lambda A$ is
\[ (\Lambda A)_r:= \hocolimsub{[r]\to [k]\leftarrow [\ell]}  A_\ell\times_{B_\ell} B_k \]
where the homotopy colimit is taken over the opposite of following category $\sE(r)$. An object is a diagram $[r]\to [k]\leftarrow [\ell]$
in $\Delta$ where the second arrow, from $[\ell]$ to $[k]$, is onto. A morphism is a commutative diagram
\[
\begin{tikzpicture}[descr/.style={fill=white}, baseline=(current bounding box.base)],
	\matrix(m)[matrix of math nodes, row sep=1.5em, column sep=1.5em,
	text height=1.5ex, text depth=0.25ex]
	{
	 {[r]} & {[k]} & {[\ell]} \\
	 {[r]} & {[k']} & {[\ell']} \\
	};
	\path[->,font=\scriptsize]
		(m-1-1) edge node [right] {} (m-1-2)
		(m-1-3) edge node [right] {} (m-2-3)
		(m-1-3) edge node [right] {} (m-1-2)
		(m-2-3) edge node [above] {} (m-2-2)
		(m-2-1) edge node [auto] {} (m-2-2)
		(m-1-2) edge node [auto] {} (m-2-2);
	\draw [double equal sign distance] (m-1-1) to (m-2-1) node [anchor=mid west] {$$};
\end{tikzpicture}
\]
in $\Delta$ (top row = source, bottom row = target). The reference map from $\Lambda A$ to $B$ is fairly obvious from
the definition of $\Lambda A$~; it is the composition
\[ \hocolimsub{[r]\to [k]\leftarrow [\ell]}  A_\ell\times_{B_\ell} B_k
\quad\lra \colimsub{[r]\to [k]\leftarrow [\ell]} B_k \quad \lra \quad B_r\,.
\]
The formula for $A^!$ is
\[ (A^!)_r= \hocolimsub{[r]\to [k]} A_k = \hocolimsub{\rule{0mm}{4mm}[r]\to [k]\xleftarrow{=} [\ell]}  A_\ell\times_{B_\ell} B_k\]
so that $A^! \subset \Lambda A$.
}
\end{defn}

\begin{defn} \label{defn-facelocsmall} {\rm Sometimes the following variant of definition~\ref{defn-faceloc} is useful.
Let $\sE_0(r)\subset \sE(r)$ be the full subcategory of $\sE(r)$ spanned by the objects $[r]\to [k]\leftarrow [\ell]$
of $\sE(r)$ where \emph{both} arrows $[r]\to [k]$ and $[k]\leftarrow [\ell]$ are surjective. Let
\[  (\Lambda^\flat A)_r:= \hocolimsub{[r]\to [k]\leftarrow [\ell]}  A_\ell\times_{B_\ell} B_k \]
where now the homotopy direct limit is taken over the opposite of $\sE_0(r)$\,. Advantage:
it is a smaller package than $(\Lambda A)_r$\,. Disadvantage: it is not so clear how $[r]\mapsto (\Lambda^\flat A)_r$
is a simplicial space. The healing observation here is that the inclusion $\sE_0(r)\to \sE(r)$ has a right adjoint.
This fact may be used to define a simplicial structure on $\Lambda^\flat A$ and it leads to a simplicial map $\Lambda A\to \Lambda^\flat A$ which can be shown to be a weak equivalence. More details are given in
\cite[\S8]{BoavidaWeissLong}.
}
\end{defn}

In \cite[\S8]{BoavidaWeissLong} it is shown that
\begin{prop}\label{prop-leftadjoint}
The functor $\Lambda$ is a homotopical left adjoint to the inclusion of the full
subcategory of the conservative objects.
\end{prop}
\begin{proof}
Let $\mathbf{C}$ be a small category and suppose that we are given an endofunctor
$\Lambda\co \mathbf{C} \to \mathbf{C}$ and a natural transformation
$\eta:\id\to\Lambda$ such that
\[ \eta_{\rule{0mm}{2.3mm}\Lambda(X)}\co \Lambda(X) \to \Lambda^2(X) \]
is an isomorphism and
\[ \Lambda(\eta_{\rule{0mm}{2.3mm}X})\co \Lambda(X) \to \Lambda^2(X) \]
is an isomorphism for all $X$ in $\mathbf{C}$.
Then $\Lambda$ is the left adjoint to the inclusion of the subcategory
$\Lambda(\mathbf{C})$, spanned by objects in the essential image of $\Lambda$, into $\mathbf{C}$.

We apply this observation to our case by taking for $\mathbf{C}$ the homotopy category of
simplicial spaces over a fixed simplicial space $B$ (or of the $\infty$-category of such),
obtained by inverting degreewise weak equivalences. Diagram~(\ref{eqn-faceloc}) describes
$\eta$, the homotopical unit of the adjunction.
The condition that $\eta_\Lambda$ is a weak
equivalence is proved in \cite[Lemma 8.7]{BoavidaWeissLong}. The condition that
$\Lambda(\eta)$ is a weak equivalence is proved in \cite[Lemma 8.8]{BoavidaWeissLong}.
The inclusion $A^!\to \Lambda A$ is a degreewise
weak equivalence if $A$ is conservative over $B$ to begin with, so that the essential image
of $\Lambda$ is the subcategory of conservative objects.
\end{proof}

It should not be taken for granted that $\Lambda A$ is a fiberwise
complete Segal space over $B$ if $A$ is a fiberwise complete Segal space over $B$. Sufficient conditions
for that are given in section~\ref{sec-locdetails} below.

In section~\ref{sec-locmodel} we give a more illuminating description of the conservatization process
which uses model category notions. Definition~\ref{defn-faceloc} remains useful for example as a tool
in the proof of theorem~\ref{thm-prods} below.

In section~\ref{sec-locmodel} we also give an illuminating description of a process which
achieves fiberwise completion and conservatization simultaneously, in a universal manner. An explicit
description is as follows. We work with simplicial spaces $A$ equipped with a reference map to a fixed simplicial set $B=N\sC$.
Write $K$ for a natural Segalization and fiberwise completion procedure. That is to say, $K$ is an endofunctor on simplicial spaces
over $B$ together with a natural transformation $\eta\co\id\to K$ such that $KA\to B$ is always a fiberwise complete
Segal space, and $\eta$ is the derived unit defining $K$ as the derived left adjoint to the inclusion of the category of fiberwise complete Segal spaces over $B$ into all simplicial spaces over $B$. (This amounts to saying that $\eta$ is a weak equivalence in the model structure for complete Segal spaces.) As mentioned above, we do not claim that $\Lambda K A$ is always a Segal space, let alone a fiberwise complete Segal space over $B$.
But let us define $(\Lambda K)^\infty(A)$ as the homotopy colimit of the diagram of natural maps
\begin{equation*}
	\begin{tikzpicture}[descr/.style={fill=white}, baseline=(current bounding box.base)] ]
	\matrix(m)[matrix of math nodes, row sep=1.0em, column sep=1.0em,
	text height=1.5ex, text depth=0.25ex]
	{
	A	&					&							&	 \\
	KA	&					&							&	 \\
(KA)^!	&	\Lambda KA		&							&	 \\
		&	K\Lambda KA		&							&	 \\
		&	(K\Lambda KA)^!	&	\Lambda K\Lambda KA		&	 \\
		&					&	K\Lambda K\Lambda KA		&	 \\
		&					&	(K\Lambda K\Lambda KA)^!	& \cdots \\
	};
	\path[->,font=\scriptsize]
		(m-1-1) edge node [auto] {$\eta$} (m-2-1)
		(m-3-1) edge node [auto] {$\simeq$} (m-2-1)
		(m-3-1) edge node [auto] {} (m-3-2)
		(m-3-2) edge node [auto] {$\eta$} (m-4-2)
		(m-5-2) edge node [auto] {$\simeq$} (m-4-2)
		(m-5-2) edge node [auto] {} (m-5-3)
		(m-5-3) edge node [auto] {$\eta$} (m-6-3)
		(m-7-3) edge node [auto] {$\simeq$} (m-6-3)
		(m-7-3) edge node [auto] {} (m-7-4);
	\end{tikzpicture}
\end{equation*}
A sequential homotopy colimit of simplicial spaces over $B$ which are conservative over $B$ is again conservative over $B$.
Essentially for that reason $(\Lambda K)^\infty A$ is conservative over $B$, for any $A\to B$ as above. By the same reasoning,
$(\Lambda K)^\infty A=(K\Lambda)^\infty(KA)$ is a fiberwise complete Segal space over $B$.
Therefore applying $(\Lambda K)^\infty$ is a way to enforce the properties \emph{conservative over $B$}  and
\emph{fiberwise complete Segal} simultaneously.

If a single application of $\Lambda K$ to $A\to B$ has an outcome $\Lambda K A\to B$ which is already a fiberwise complete Segal space over $B$ for whatever reasons, then there is no need to inflict $\Lambda K$ again and again; in other words the inclusion $(\Lambda K)A\to (\Lambda K)^\infty A$
is then a degreewise weak equivalence. This is a special feature, but it turns out to be shared by configurations categories as we will see in section \ref{sec-locdetails}.

\begin{defn} \label{defn-unwieldy} {\rm Suppose that $X$ and $Y$ are conservative over $N\fin$, in addition to being fiberwise
complete Segal spaces over $N\fin$. Let
\[ X\boxtimes Y:= (\Lambda K)^\infty(X\boxtimes^\pre Y) \]
be the simplicial space over $N\fin$ obtained from $X\boxtimes^\pre Y$ by inflicting conservatization
and fiberwise complete Segalization simultaneously on $X\boxtimes^\pre Y$.
}
\end{defn}

\begin{expl} \label{expl-conftens} {\rm Take $X=\config(M)$, $Y=\config(N)$ and $Z=\config(M\times N)$. These are
all conservative and fiberwise complete over $N\fin$. Form $X\boxtimes^\pre Y$
as in definition~\ref{defn-pretensor}. There is an easy map from what is essentially
$X\boxtimes^\pre Y$ to $Z$ over $N\fin$. Here is a semi-detailed description. We use the particle models, so that
$X$, $Y$ and $Z$ are the nerves of certain topological categories $\sA$\,, $\sB$ and $\sC$,
respectively. Then $X\boxtimes^\pre Y$ is also the nerve of a topological category (over $\fin$) which we
called $\sA\bar\times \sB$ in example~\ref{expl-config}. Therefore we should be looking for a continuous functor
from $\sA\bar\times \sB$ to $\sC$. Clearly an object
\[ (f\co \uli r\to M,\,g\co \uli s\to N,\,u\co \uli k \to \uli r\times \uli s) \]
of $\sA\bar\times \sB$ (notation of example~\ref{expl-config})
determines an object of $\sC$ given by
\[ (f\times g)u\co \uli k\to M\times N\,. \]
Roughly the same prescription works for morphisms. A morphism
\[  \Gamma\co (f^0,g^0,u^0) \lra (f^1,g^1,u^1) \]
in $\sA\bar\times\sB$ determines a morphism
from $(f^0\times g^0)u^0$ to $(f^1\times g^1)u^1$ in $\sC$ \emph{provided}
the morphisms $f^0\to g^0$ and $f^1\to g^1$ in $\sA$ and $\sB$ determined by $\Gamma$, respectively,
are parameterized by the same interval $[0,a]$. This is then an extra condition that we must add;
so we obtain a continuous functor from a certain subcategory (nameless) of $\sA\bar\times \sB$ to $\sC$. The inclusion of the nerve of that subcategory in the nerve of $\sA\bar\times\sB$ is of course a weak equivalence.}
\end{expl}

The functor $(\Lambda K)^\infty$ is a derived left adjoint to the inclusion functor (of conservative, fiberwise complete
Segal spaces over $\fin$ into all simplicial spaces over $\fin$), as will be explained in section \ref{sec-locmodel}. Hence,
by its universal property, and since $Z$ is fiberwise complete and conservative over $N\fin$,
the map from what is essentially $X\boxtimes^\pre Y$ to $Z$ must have a unique (in the derived sense) factorization
through $X\boxtimes Y$.

\begin{thm} \label{thm-prods} The resulting map $\config(M)\boxtimes\config(N)\to \config(M\times N)$ is a weak equivalence. Indeed, the induced map  
\[ \Lambda(\config(M)\boxtimes^\pre\config(N))\lra \Lambda(\config(M\times N))\simeq\config(M\times N) \]
is already a degreewise weak equivalence. 
\end{thm}

This will be proved in section~\ref{sec-locdetails} after some preparations. But we can make a start here by proving a much
weaker statement.

\begin{lem} \label{lem-firststep} The resulting map $\Lambda(\config(M)\boxtimes^\pre\config(N))\lra \config(M\times N)$ is a weak equivalence
in simplicial degree $0$.
\end{lem}

In order to prove this lemma, we will use a result about stratified spaces. For a homotopically stratified space $Q$, let $\sE\sP_Q$ denote the (reverse) exit path category of $Q$. This is a topological category. The object space of $\sE\sP_Q$ is by definition the topological disjoint union of the strata of $Q$, where each stratum has its own topology as a subspace of $Q$. The morphism space is the space of \emph{reverse} exit paths in $Q$. There is an inclusion
\[
\iota : \sE\sP_Q \to \sP_Q
\]
into the full path category of $Q$.
\begin{prop}\label{prop:hstrat}
Let $Q$ be a homotopically stratified space which is locally contractible. Then $\iota$ induces a weak equivalence
$|N \sE\sP_Q| \rightarrow |N \sP_Q| \simeq Q$.
\end{prop}
Here the vertical bars denote geometric realization or homotopy colimit over $\Delta$ (the two agree since we are free to take spaces to mean simplicial sets in our context). This appears in \cite[Corollary 9.3]{Miller2} and also \cite[Corollary A.10.4]{Lurie} and \cite[Corollary 4.6]{Woolf}).

\proof[Proof of Lemma \ref{lem-firststep}] The ordered configuration spaces of $M\times N$ have a natural stratification related to the product structure. Indeed,
an embedding $a\co\uli k\to M\times N$ determines maps $\uli k\to M$ and $\uli k\to N$ which can fail to be injective. These two
maps have unique factorizations
\[
\uli k \lra \uli r \lra M \quad , \quad \uli k \lra \uli s \lra N \; ,
\]
of type \emph{selfic (surjective) map followed by injective map}.
Therefore we have a stratification of $Q=\emb(\uli k\,,\,M\times N)$ where the strata correspond to
certain pairs of selfic maps $(\uli k\to \uli r,\,\uli k\to \uli s\,)$\,. Now we apply proposition \ref{prop:hstrat} to this stratification to arrive at a weak equivalence
\[  \hocolimsub{[r]} (N\sE\sP_Q)_r ~\lra~\hocolimsub{[r]} (N \sP_Q)_r ~. \]
Unraveling definition~\ref{defn-facelocsmall} (of $\Lambda^\flat$) and using the description of $X\boxtimes^\pre Y$ from example \ref{expl-config}, we recognize this as the map of spaces
\begin{itemize}
\item[-] induced by $X\boxtimes^\pre Y\to Z$ as in example~\ref{expl-conftens}
\item[-] from the part of
$(\Lambda^\flat (X\boxtimes^\pre Y))_0$ sitting over $\uli k\in (N\fin)_0$
\item[-] to the part of $(\Lambda^\flat Z)_0 \simeq Z_0$
sitting over $\uli k\in (N\fin)_0$. \qed
\end{itemize}

\section{Property beta} \label{sec-locdetails}
Let $\sC$ be a small (discrete) category.

\begin{lem} \label{lem-faceloccplt} If $A$ is a fiberwise complete Segal space over $B=N\sC$ and if $\Lambda A$ is a Segal space, then
$\Lambda A$ is also fiberwise complete over $B$.
\end{lem}

\proof See \cite[\S8]{BoavidaWeissLong}.

\begin{defn} \label{defn-transp} \cite[\S8]{BoavidaWeissLong}. {\rm Suppose that $w\co A\to B$ makes $A$ into a fiberwise
complete Segal space over $B$. We say that $w\co A\to B$ has
\emph{property beta} if, for $f\in A_1$ with source $x=d_0f\in A_0$ and target $y=d_1f\in A_0$
such that $w(f)\in B_1=(N\sC)_1$ is an identity morphism, the map
$\Lambda(A\overcat x) \to \Lambda(A\overcat y)$
induced by $f$ is a weak equivalence in degree 0.
}
\end{defn}

\begin{thm} \label{thm-transp} Under conditions as in definition~\emph{\ref{defn-transp}}, if $A\to B$ has property beta,
then $\Lambda A$ is a fiberwise complete Segal space over $B$.
Moreover, for any $z\in A_0$ there is a canonical simplicial map
from $\Lambda(A\overcat z)$ to $(\Lambda A\overcat z)$
which is a weak equivalence in degree $0$.
\end{thm}

\proof See \cite[\S8]{BoavidaWeissLong}.

\begin{rem} \label{rem-pinought} \cite[\S8]{BoavidaWeissLong}. {\rm In diagram~(\ref{eqn-faceloc}), the map from
$\pi_0 A^!_0\cong \pi_0 A_0$ to $\pi_0((\Lambda A)_0)$
determined by the horizontal arrows is surjective. This can be seen from the definition
\[  (\Lambda A)_0= \hocolimsub{[0]\to[k]\leftarrow[\ell]} A_\ell\times_{B_\ell} B_k~. \]
Every object $[0]\to[k]\leftarrow[\ell]$ of the indexing category is the target of some
morphism from the object $[0]\to [0]\leftarrow [0]$.
}
\end{rem}

Throughout, we use the notation of example~\ref{expl-conftens}. Let $W=X\boxtimes^\pre Y$.

\begin{lem}\label{lem-beta}
The map $W \to N\fin$ has property beta.
\end{lem}
\begin{proof}
Choose $w \in W_0$ and let $x \in X_0$ and $y \in Y_0$ be the images of $w$. We saw in
remark \ref{rem-comma} that $W/w$ is a full Segal subspace of $X/x \boxtimes^\pre Y/y$. By using the
formula defining $\Lambda^\flat$ it follows that the induced map
\[
\Lambda(W/w) \to \Lambda(X/x \boxtimes^\pre Y/y)
\]
is, in degree 0 and up to weak equivalence, an inclusion of path components. In more detail, $(\Lambda^\flat(X/x \boxtimes^\pre Y/y))_0$ is the homotopy colimit of a simplicial space with space of $n$-simplices given by the pullback
\[
N\fin_0 \times_{N\fin_n} (X/x \boxtimes^\pre Y/y)_n \; .
\]
By remark~\ref{rem-comma}, especially the last sentences, this simplicial space is the disjoint union of two simplicial subspaces. One of these simplicial subspaces is the one whose homotopy colimit defines $(\Lambda^\flat(W/w))_0$.

By \cite[Corollary 3.6]{BoavidaWeissLong}, the Segal spaces $X/x$ and $Y/y$ are weakly equivalent to configuration categories $\config(U)$ and $\config(V)$ where $U\subset M$ and $V\subset N$ are small tubular neighborhoods of $x$ and $y$, respectively. It follows from lemma \ref{lem-firststep} that the map
\[
\Lambda(\config(U) \boxtimes^\pre \config(V)) \lra \config(U \times V)
\]
is a weak equivalence in degree 0. Let $T\subset U\times V$ be the union of the path components determined by
$w$. More precisely, $w$ determines an object of $\boxfin$, alias diagram $\uli r \leftarrow \uli k \to \uli s$
in $\fin$, such that the map $\uli k\to \uli r\times\uli s$ is injective; here $\uli r$ is identified with
$\pi_0(U)$ and $\uli s$ is identified with $\pi_0(V)$, so that $\uli k$ is identified with a subset of $\pi_0(U\times V)$.
But the gist of the above observations relying on remark~\ref{rem-comma}
is that we have identified $(\Lambda(W/w))_0$ with the homotopy pullback of
\[
(\Lambda(\config(U) \boxtimes^\pre \config(V)))_0 \lra \config(U \times V)_0 \longleftarrow \config(T)_0 \,.
\]
Therefore the degree $0$ part of $\Lambda(W/w)$ is now identified with the degree 0 part of $\config(T)$,
up to weak equivalence.

Given a morphism $w \to w^\prime$ in $W$ over an identity morphism in $\fin$, the tubular neighborhoods $U,V,U',V'$ and $T,T'$ may be chosen (up to replacing $w\to w'$ by another element in the same path component of $W_1$)
so that $U\subset U'$, $V\subset V'$, $T\subset T'$, and the induced map
$\Lambda(W/w) \to \Lambda(W/w^\prime)$
is, in degree 0 and up to weak equivalence, identified with the map $\config(T)\to \config(T')$.
Therefore it is a weak equivalence in degree 0. Note in passing that $U'$ can have fewer path components than $U$,
and $V'$ can have fewer path components than $V$, but the inclusion $T\to T'$ induces a bijection on $\pi_0$
(since $w\to w'$ covers an identity morphism in $\fin$).
\end{proof}

\begin{proof}[Proof of theorem~\ref{thm-prods}]
It was already shown in lemma~\ref{lem-firststep}
that our map from $\Lambda W$ to $Z=\config(M\times N)$ is a
weak equivalence in degree $0$. Since $W \to N \fin$ has property beta, we know that $\Lambda W$ is a Segal space. Therefore all we need to show is that the map $\Lambda W \to Z$ is a weak equivalence in degree $1$.

Consider the square
\begin{equation*}
	\begin{tikzpicture}[descr/.style={fill=white}, baseline=(current bounding box.base)] ]
	\matrix(m)[matrix of math nodes, row sep=2.5em, column sep=2.5em,
	text height=1.5ex, text depth=0.25ex]
	{
	(\Lambda W)_1 & Z_1 \\
	(\Lambda W)_0 & Z_0\\
	};
	\path[->,font=\scriptsize]
		(m-1-1) edge node [auto] {} (m-1-2)
		(m-1-1) edge node [left] {target} (m-2-1)
		(m-1-2) edge node [auto] {target} (m-2-2)
		(m-2-1) edge node [auto] {} (m-2-2);
	\end{tikzpicture}
\end{equation*}
For $w \in W_0$ there is the induced map between vertical homotopy fibers
\[
( \Lambda W / w)_0 \to (Z/z)_0
\]
where $z$ is the image of $w$. It is easy to see that $Z/z$ is weakly equivalent to the configuration category $\config(T)$ where $T$ is a certain union of components of $U \times V$ determined by $w$, as in the proof of lemma~\ref{lem-beta}.
By construction, the map $\Lambda(W/w) \to \config(U \times V)$ factors through $\config(T)$ and this factorization
gives a weak equivalence $\Lambda(W/w)_0\to \config(T)_0$ as seen in the proof of lemma~\ref{lem-beta}.

Applying lemma \ref{lem-beta} and theorem \ref{thm-transp}, $( \Lambda W/w)_0 \simeq \Lambda(W/w)_0$, and so we have that the map between vertical homotopy fibers in the square (over $w \in W_0$) is a weak equivalence. By remark \ref{rem-pinought}, we conclude that the square is homotopy cartesian. Therefore  $\Lambda W \to Z$ is a weak equivalence in degree $1$.
\end{proof}

As the reader will have noticed, theorem~\ref{thm-prods} is a more precise reformulation of theorem~\ref{thm:main}.
We can tick off theorem~\ref{thm:main} and turn to the proof theorem~\ref{thm:mainorb}.

\begin{proof}[Proof of theorem~\emph{\ref{thm:mainorb}}] Write $K=\TOP(M)$ and $L=\TOP(N)$. With the particle models
for $\config(M)$ and $\config(N)$, the actions of $K$ on $\config(M)$ and of $L$ on $\config(N)$ that
we require are obvious. There is a homotopy fiber sequence
\[ \config(M)\boxtimes^\pre \config(N) \lra \config(M)_{hK} \boxtimes^\pre \config(N)_{hL} \lra B(K\times L)\times N\fin \]
of simplicial spaces over $N\fin$. Here $B(K\times L)\times N\fin$ is fiberwise constant over $N\fin$.
Applying $\Lambda$ to each term preserves the homotopy fiber sequence status. (Indeed, $\Lambda$
was defined in terms of homotopy colimit constructions and homotopy colimits are stable under homotopy base change.) Therefore we obtain a homotopy fiber sequence
\[
\Lambda(\config(M)\boxtimes^\pre \config(N)) \lra \Lambda(\config(M)_{hK} \boxtimes^\pre \config(N)_{hL}) \lra B(K\times L)\times N\fin \]
of simplicial spaces over $N\fin$. This has several good consequences; in particular we can deduce that
$\Lambda(\config(M)_{hK} \boxtimes^\pre \config(N)_{hL})$ is fiberwise complete Segal over $N\fin$. --- The map
\[ \config(M)_{hK} \boxtimes \config(N)_{hL} \lra \config(M \times N)_{h(K\times L)} \]
that we require will now be constructed like the map in theorem~\ref{thm-prods}. We begin by observing that $\config(M)_{hK}$ and $\config(N)_{hL}$ are fiberwise complete Segal spaces over $N\fin$. (For the Segal property, this hinges on fact that taking homotopy orbits commutes with homotopy pullbacks.) So, by proposition \ref{prop-preboxcomplete}, it follows that
\[  \config(M)_{hK} \boxtimes^\pre \config(N)_{hL} \]
is a fiberwise complete Segal space over $N\fin$.
There is a map
\[ \config(M)_{hK} \boxtimes^\pre \config(N)_{hL} \lra \config(M \times N)_{h(K\times L)} \]
which can be constructed by a mild adaptation of example~\ref{expl-conftens}. It is a map over $N\fin$ and so we can
apply conservatization $\Lambda$ to it. The effect on the right-hand side is negligible since that
was already conservative over $N\fin$. Regarding the left-hand side, we know already that $\Lambda$ applied to it is
fiberwise complete Segal over $N\fin$; therefore we may call it $\config(M)_{hK} \boxtimes \config(N)_{hL}$
and we now have our map
\[   \config(M)_{hK} \boxtimes \config(N)_{hL} \lra  \config(M \times N)_{h(K\times L)}\,. \]
The reasoning so far tells us also that this map is the middle column in a commutative diagram
\[
\begin{tikzpicture}[descr/.style={fill=white}, baseline=(current bounding box.base)],
	\matrix(m)[matrix of math nodes, row sep=2.5em, column sep=2.5em,
	text height=1.5ex, text depth=0.25ex]
	{
	  \config(M) \boxtimes \config(N) &  \config(M)_{hK} \boxtimes \config(N)_{hL} &  B(K\times L) \\
	 \config(M \times N) & \config(M \times N)_{h(K\times L)} &  B(K\times L) \\
	};
	\path[->,font=\scriptsize]
		(m-1-1) edge node [right] {} (m-1-2)
		(m-1-2) edge node [right] {} (m-1-3)
		(m-1-1) edge node [right] {} (m-2-1)
		(m-2-2) edge node [above] {} (m-2-3)
		(m-2-1) edge node [auto] {} (m-2-2)
		(m-1-2) edge node [auto] {} (m-2-2);
	\draw [double equal sign distance] (m-1-3) to (m-2-3) node [anchor=mid west] {$$};
\end{tikzpicture}
\]
where the rows are homotopy fiber sequences. Since the left column is a weak equivalence, the middle column is a weak
equivalence, too. \end{proof}

\section{Fibrant replacement solves some problems} \label{sec-locmodel}
The conservatization functor $\Lambda$ is related to the identity functor via a zig-zag.
So it is not quite a fibrant replacement in the appropriate model category, though for all purposes it serves as such.
In this section, we make this statement a little more explicit. One reason for having model category formulations is that universal properties are easier to formulate and employ, in that they do not involve zig-zags.

Throughout this section, $B$ denotes a fixed simplicial (discrete) space which is the nerve of a small category. The category of simplicial spaces over
$B$ has two model structures, the injective and the projective, where the weak equivalences are given degreewise (forgetting the
reference maps to $B$). We refer to each of these generically as the degreewise model structure. Starting from the degreewise model
structure, we can localize to obtain a new model structure where the fibrant objects are the conservative simplicial spaces over
$B$ (which are moreover fibrations in the degreewise model structure). This is the left Bousfield localization at the maps of the form $\Delta[k] \xrightarrow{} \Delta[\ell] \xrightarrow{}  B$ where the first map is induced by a surjection in $\Delta$.

\begin{prop}
There is a zigzag of degreewise weak equivalences between $\Lambda(X)$ and $\mathbf{\Lambda}(X)$, natural in $X$, where
$\mathbf{\Lambda}$ is any fibrant replacement functor in the model structure for conservative simplicial spaces over $B$.
\end{prop}
\begin{proof}
In the diagram
\[
\mathbf{\Lambda}(X) \gets \mathbf{\Lambda}(X)^! \to \Lambda(\mathbf{\Lambda}(X))
\]
both arrows are (degreewise) weak equivalences; the right-hand one is so because $\mathbf{\Lambda}(X) \to B$ is conservative.
Moreover, the map $\Lambda(X) \to \Lambda(\mathbf{\Lambda}(X))$ is a weak equivalence between conservative objects (as one can
check by mapping it to any other conservative object and using that $\Lambda$ is a homotopical left adjoint by proposition \ref{prop-leftadjoint}), and hence a degreewise weak equivalence.
\end{proof}

\begin{defn} {\rm
For a simplicial space $X$ over $B$, let $L(X)$ be the simplicial space defined as the homotopy colimit of
\[
X \to K(X) \to \mathbf{\Lambda}K(X) \to K\mathbf{\Lambda}K(X) \to \mathbf{\Lambda}K\mathbf{\Lambda}K(X) \to \cdots
\]
(over $B$). Clearly, $L(X)$ is naturally weakly equivalent to $(\Lambda K)^{\infty}(X)$.}
\end{defn}

In what follows, we call a simplicial map $X \to B$ \emph{local} if it is a conservative, fiberwise complete Segal space (and the
map is a fibration in the degreewise model structure). There is a model structure on the category of simplicial spaces over $B$
where the fibrant objects are the local ones. The generating trivial cofibrations in the degreewise (meaning, projective or injective)
model structure are also generating trivial cofibrations for this model structure. In addition, the maps in the category of
simplicial spaces over $B$:
\begin{itemize}
\item
$
\{ \Delta[0] \to E \to B \}
$
, where $E$ denotes the groupoid with two objects $x, y$ and exactly two non-identity isomorphisms $x \to y$ and $y \to x$,
\item
$
 \{ S(n) \to \Delta[n] \to B: n \geq 2\}
$
, where $S(n)$ denotes the homotopy colimit of
$$
\Delta[1] \xleftarrow{d_1} {\Delta[0]} \xrightarrow{d_0} \dots \xleftarrow{d_1} {\Delta[0]}  \xrightarrow{d_0} \Delta[1] \; ,
$$
\item
$
\{ \Delta[k] \xrightarrow{f^*} \Delta[\ell] \to B : f \mbox{ surjective } \} \, ,
$
\end{itemize}
are the newly added generating trivial cofibrations (note that these maps are not cofibrations in the projective sense,
so in that case we replace each of them by cofibrations between cofibrant objects first). These account for the completeness
condition, Segal condition, and conservativity condition, respectively.
\begin{prop}\label{prop:fibrepl}
The functor $L$ is a fibrant replacement. In other words:
\begin{enumerate}
\item $L(X) \to B$ is local, and
\item the inclusion $X \to L(X)$ is a local weak equivalence, that is, for every local $Z \to B$, the induced map
\[
\RR \map_B(L(X), Z) \to \RR \map_B(X, Z)
\]
is a weak equivalence.
\end{enumerate}
\end{prop}
\begin{proof}
We begin with (2). Let $L^i(X)$ denote the $i$-th space in the tower defining $L(X)$. It suffices to show that for every $i \geq 0$ and
every local $Z \to B$, the map
\[
\RR \map_{B}(L^i(X), Z) \to \RR \map_{B}(X, Z)
\]
is a weak equivalence. Arguing inductively, this amounts to showing that the maps
\[
\RR \map_{B}(\mathbf{\Lambda}K X, Z) \to \RR \map_{B}(KX, Z) \to \RR \map_{B}(X, Z)
\]
are weak equivalences. And this is clear, because $Z$ is local.

The proof of (1) is essentially an application of the small object argument. We want to show that for every generating trivial
cofibration $\iota : A \to C$ (over $B$) and map $\gamma : A \to L(X)$ (over $B$), there is a lift as pictured:
\begin{equation*}
	\begin{tikzpicture}[descr/.style={fill=white}, baseline=(current bounding box.base)] ]
	\matrix(m)[matrix of math nodes, row sep=2.0em, column sep=2.0em,
	text height=1.5ex, text depth=0.25ex]
	{
	A & L(X) \\
	C & \\
	};
	\path[->,font=\scriptsize]
		(m-1-1) edge node [auto] {$\gamma$} (m-1-2)
		(m-1-1) edge node [left] {$\iota$} (m-2-1);
	\path[->, dashed,font=\scriptsize]
		(m-2-1) edge node [auto] {} (m-1-2);
	\end{tikzpicture}
\end{equation*}
Our generating trivial cofibrations are either trivial cofibrations in the model structure
built on the concept of conservative simplicial space over $B$, or they are trivial cofibrations in the model structure built on the concept of fiberwise complete Segal space over $B$. In the first case, we can choose a
factorization of the map $\gamma$ through a map of the form
$A \to (\mathbf{\Lambda}K)^i X$. Then such a lift (to $(\mathbf{\Lambda}K)^i X$) exists by definition. In the second case
we can choose factorization of the map $\gamma$ through a map of the form
$A \to K(\mathbf{\Lambda}K)^i X$. Then such a lift exists by definition.
Thus, $L(X)$ has the right lifting property with respect to all generating trivial
cofibrations.
\end{proof}

\begin{cor}\label{cor:fibrepl} Suppose that $X \to B$ is such that $\Lambda K(X) \to B$ is already local (i.e. a conservative,
fiberwise complete Segal space over $B$). Then the inclusion $\mathbf{\Lambda}K(X) \to L(X)$ is a degreewise weak equivalence.
\end{cor}
\begin{proof}
The inclusion $\mathbf{\Lambda}K(X) \to L(X)$ is a local weak equivalence, and a map between local objects is a local weak equivalence
if and only if it is a degreewise weak equivalence.
\end{proof}

\section{Truncation}
Let $\fin_{\le k}$ be the full subcategory of $\fin$ made up by the objects $\uli \ell$ where $\ell=0,1,2,\dots,k$. The word \emph{truncation}, as used in this section, generally refers to a procedure for restricting something from $\fin$ to $\fin_{\le k}$ for a fixed $k$.

\begin{defn} {\rm The level $k$ \emph{truncation} of a simplicial space $X$ over $N\fin$ is the degreewise pullback $X^{k]}$ of
\[
	\begin{tikzpicture}[descr/.style={fill=white}]
	\matrix(m)[matrix of math nodes, row sep=2.5em, column sep=2.5em,
	text height=1.5ex, text depth=0.25ex]
	{
		 & X  \\
	N(\fin_{\le k}) & N\fin \\
	};
	\path[->,font=\scriptsize]
		(m-2-1) edge node [auto] {} (m-2-2)
		(m-1-2) edge node [auto] {} (m-2-2);
	\end{tikzpicture}
\]
}
\end{defn}

The tensor product $\boxtimes$ has a variant where we begin with fiberwise complete Segal
spaces $X$ and $Y$ over $N(\fin_{\le k})$ and define $X\boxtimes Y$ to be a fiberwise complete and conservative
Segal space over $N(\fin_{\le k})$.

\begin{prop}\label{prop-boxandtruncation}
The construction $\boxtimes$ commutes with truncation, i.e.
$$(X \boxtimes Y)^{k]} \simeq X^{k]} \boxtimes Y^{k]}$$
for any two simplicial spaces $X$ and $Y$ over $N\fin$.
\end{prop}

The proof of this proposition relies on the following.

\begin{lem}\label{lem-commutetruncation}
Let $W$ be a simplicial space over $\fin$. Then
\begin{itemize}
\item[(i)] $\mathbf{\Lambda}(W^{k]}) \simeq (\mathbf{\Lambda} W)^{k]}$ \,,
\item[(ii)] $K(W^{k]}) \simeq (K W)^{k]}$\,.
\end{itemize}
\end{lem}
\begin{proof}
Let $\Phi$ be $\mathbf{\Lambda}$ or $K$. We view $\Phi$ as an endofunctor on the category of simplicial spaces over $B$ where $B$ is $\fin$ or $\fin_{\le k}$, depending on the context. For the duration of this proof, we call a map $X \to Y$ of simplicial spaces over a simplicial space $A$ a \emph{local} weak equivalence if $\Phi X \to \Phi Y$ is a degreewise weak equivalence over $B$. An object $Z \to B$ will  be called \emph{local} if it there is a zigzag of degreewise weak equivalences (over $B$) between $Z$ and $\Phi(X)$ for some $X \to B$. In other words, if $\Phi = K$, local means fiberwise complete; if $\Phi = \mathbf{\Lambda}$, local means conservative.

Let $\tau$ denote the inclusion of $\fin_{\le k}$ in $\fin$ and $\tau^*$ the pullback along $\tau$. By construction, the map $X \to \Phi X$ over $\fin$ is a local weak equivalence. Now consider the following square
\begin{equation*}
	\begin{tikzpicture}[descr/.style={fill=white}, baseline=(current bounding box.base)] ]
	\matrix(m)[matrix of math nodes, row sep=2.5em, column sep=2.5em,
	text height=1.5ex, text depth=0.25ex]
	{
	\tau^* X & \tau^* \Phi X \\
	\Phi \tau^* X & \Phi \tau^* \Phi X \; .\\
	};
	\path[->,font=\scriptsize]
		(m-1-1) edge node [auto] {} (m-1-2)
		(m-1-1) edge node [left] {} (m-2-1)
		(m-1-2) edge node [auto] {} (m-2-2)
		(m-2-1) edge node [auto] {} (m-2-2);
	\end{tikzpicture}
\end{equation*}
The vertical maps are local weak equivalences by construction. We claim that the top horizontal arrow is a local weak equivalence. Taking this for granted for the moment, it follows that the lower horizontal arrow is also a weak equivalence.
Since $\tau^*$ sends local objects to local objects, the maps into the lower right hand corner are local weak equivalences between local objects and so must be degreewise weak equivalences (over $\fin_{\le k}$). That is, $\Phi \tau^*X \simeq \tau^* \Phi X$.

We now turn to the proof of the claim. We will show that if $X$ and $Y$ are two simplicial spaces over $\fin$ and $f : X \to Y$ is a map over $\fin$ which is also a local weak equivalence, then $\tau^* f$ is a local weak equivalence. The restriction functor $\tau^*$ has a right adjoint $\tau_*$. For a simplicial space $W$ over $\fin_{\le k}$ and an $n$-simplex $\sigma$ of $N\fin$, the fiber of $\tau_* W$ over $\sigma$ is identified with the fiber of $W$ over $\tau^*\sigma$ if $\tau^*\sigma$ is non-empty, and otherwise is a point. Here $\tau^* \sigma$ denotes the pullback of $\sigma$ along $\tau$. If among the $n+1$ vertices of $\sigma$ exactly $\ell+1$ are in $\fin_{\leq k}$, then $\tau^* \sigma$ is an $\ell$-simplex in $N(\fin_{\le k})$. Using this description of $\tau_*$ it can be checked, for the fiberwise completeness and conservative conditions separately, that if $W$ is a local simplicial space over $\fin_{\le k}$ then $\tau_* W$ is a local simplicial space over $\fin$.

The map $\tau^* f$ is a local weak equivalence if and only if for every local simplicial space $W$ over $\fin_{\le k}$, the restriction map
\[
\RR \map_{\fin_{\le k}}(\tau^*Y, W) \to \RR \map_{\fin_{\le k}}(\tau^*X, W)
\]
is a weak equivalence. (Here the derived mapping spaces are taken with respect to degreewise weak equivalences.) By adjunction, this map is weakly equivalent to
\[
\RR \map_{\fin}(Y,\tau_*W) \to \RR \map_{\fin}(X, \tau_*W) \;.
\]
But, since $\tau_* W$ is local and $f$ is a local weak equivalence, this map is a weak equivalence. Therefore, $\tau^* f$ is a weak equivalence.
\end{proof}

\begin{proof}[Proof of Proposition \ref{prop-boxandtruncation}]
As before, write $Z$ for $X \boxtimes^\pre Y$. Since homotopy colimits are stable under homotopy base change, the simplicial space $L(Z)^{k]}$ is weakly equivalent to the homotopy colimit of
\[
Z^{k]} \to K(Z)^{k]} \to \mathbf{\Lambda}K(Z)^{k]} \to K\mathbf{\Lambda}K(Z)^{k]} \to
\mathbf{\Lambda}K\mathbf{\Lambda}K(Z)^{k]} \to \cdots
\]
(over $\fin_{\le k}$). Using lemma \ref{lem-commutetruncation}, it follows that $L(Z)^{k]}$ is weakly equivalent to $L(Z^{k]})$. This proves the result, since $X^{k]}\boxtimes^\pre Y^{k]}=(X\boxtimes^\pre Y)^{k]}$.
\end{proof}

\section{Relations and speculations involving the little disks operad}

In \cite{BoavidaWeissLong}, we showed that the little $d$-disks operad $E_d$ can be fully reconstructed from the configuration category of $\RR^d$. Let us briefly recall how this goes. A simplicial space over the nerve of $\fin$ is the same as a contravariant functor from $\simp(\fin)$, the simplex category of $N\fin$, to spaces. An object of $\simp(\fin)$ is a string of maps of finite sets. Viewing such a string as a rooted tree (with no leaves), we obtain a functor $j$ from $\simp(\fin)$ to $\mathsf{Tree}$, the category of trees of \cite{MoerdijkWeiss}. Then, by pullback, $j$ determines a functor $j^*$ from $\mathsf{Tree}$-spaces (alias dendroidal spaces) to $\simp(\fin)$-spaces.

Let $\sC$ denote the full $\infty$-subcategory of the category of dendroidal spaces spanned by those which satisfy a Segal condition and have contractible spaces of $0$ and $1$-arity operations. In \cite[section 7]{BoavidaWeissLong}, we proved the following.
\begin{thm}
The restriction of $j^*$ to $\sC$ is homotopically fully faithful.
\end{thm}
Moreover, $j^* E_d \simeq \config(\RR^d)$ and so the theorem implies that the derived counit
\[
\LL j_! \config(\RR^d) \to E_d
\]
is a weak equivalence of dendroidal spaces (in a model structure in which the fibrant objects are precisely the objects of $\sC$). An immediate corollary of Theorem \ref{thm:main} is then that $j^* E_d \boxtimes j^* E_\ell \simeq j^* E_{d + \ell}$.

\medskip
We end with a question. Suppose $P$ and $Q$ are operads with contractible spaces of $0$ and $1$-ary operations and let $P \otimes Q$ denote their (homotopical) Boardman-Vogt tensor product.
Under what conditions is there a natural weak equivalence
\[
j^*(P) \boxtimes j^*(Q) \to j^*(P \otimes Q) \; ?
\]

What is meant by weak equivalence here is also up for negotiation. We know of two cases for which the answer is yes, and the meaning of weak equivalence is degreewise. Assuming theorem \ref{thm:main}, this is the case when $P$ and $Q$ are little disks operads. This follows easily from the additivity theorem and is in fact equivalent to it. The answer is also yes when one of the operads, $P$ say, is the operad with only two operations: one in degree $0$ and one in degree $1$ (the identity). In that case, $P$ is the unit for the Boardman-Vogt product (in the context where all operads have contractible spaces of $0$ and $1$ ary operations). But $j^*P$ corresponds to the map of simplicial spaces $\Delta[1] \to N\fin$ corresponding to the unique map $\uli 0 \to \uli 1$, and it can be checked directly that $j^*Q \boxtimes^{\pre} j^*P \simeq j^*Q$.

\end{document}